\documentclass[12pt,a4paper]{amsart}
\usepackage{amsmath}
\usepackage{drpack}
\usepackage[english]{babel}
\usepackage{verbatim}
\usepackage[shortlabels]{enumitem}
\usepackage{color}
\usepackage{tikz-cd}

\newtheoremstyle{mio}%
{}{} 
{\itshape}{} 
{\bfseries}{.}{ } 
{#1 #2\thmnote{~\mdseries(#3)}} 
\theoremstyle{mio}
\newtheorem{teor}{Theorem}[section]
\newtheorem{cor}[teor]{Corollary}
\newtheorem{prop}[teor]{Proposition}
\newtheorem{lemma}[teor]{Lemma}
\newtheorem{defin}[teor]{Definition}

\newtheoremstyle{definition2}%
{}{} 
{}{} 
{\bfseries}{.}{ } 
{#1 #2\thmnote{\mdseries~ #3}} 
\theoremstyle{definition}
\newtheorem{ex}[teor]{Example}
\newtheorem{oss}[teor]{Remark}
\newtheorem*{quest}{Question}

\title{Boundness in almost Dedekind domains}
\author{Dario Spirito}
\date{\today}
\address{Dipartimento di Scienze Matematiche, Informatiche e Fisiche, Universit\`a degli Studi di Udine, Udine, Italy}
\email{dario.spirito@uniud.it}
\subjclass[2010]{13F05; 13A15.}
\keywords{Almost Dedekind domains; critical maximal ideals; SP-domains; free abelian group; invertible ideals}

\newcommand{\inscrit}{\mathrm{Crit}}
\newcommand{\Jac}{\mathrm{Jac}}

\newcommand{\V}{\mathcal{V}}
\newcommand{\D}{\mathcal{D}}
\newcommand{\mmax}{\mathcal{M}}
\newcommand{\Inv}{\mathrm{Inv}}
\newcommand{\Div}{\mathrm{Div}}
\newcommand{\critx}[1]{\phantom{}^{#1}\inscrit}
\newcommand{\bcrit}{\critx{\omega}}

\begin{document}

\begin{abstract}
We study different form of boundness for ideals of almost Dedekind domains, generalizing the notions of critical ideals, radical factorization, and SP-domains. We show that every almost Dedekind domain has at least one noncritical maximal ideals and, indeed, the set of noncritical maximal ideals is dense in the maximal space, with respect to the constructible topology; as a consequence, we show that every almost Dedekind domain is SP-scattered, and in particular that the group $\Inv(D)$ of invertible ideals of an almost Dedekind domain $D$ is always free. If $D$ is an almost Dedekind domain with nonzero Jacobson radical, we also show that there is at least one element whose ideal function is bounded.
\end{abstract}

\maketitle

\section{Introduction}
Dedekind domains are one of the most basic rings in commutative algebra. One of their many characterization is about factorization of ideals: an integral domain is Dedekind if and only if every ideal can be written (uniquely) as a finite product of prime ideals. One generalization of this property is requiring a factorization in \emph{radical} ideals: when every ideal has a radical factorization, the domain is said to be an \emph{SP-domain}. SP-domain are not very far from being Dedekind domains: indeed, they are \emph{almost Dedekind domains}, meaning that the localization at every maximal ideal is a discrete valuation ring (equivalently, a domain is almost Dedekind if it is locally Dedekind) \cite[Theorem 2.4]{vaughan-SP}. This property implies several interesting consequences: for example, the group of invertible ideals of an SP-domain $D$ is isomorphic to the group of continuous functions of compact support from the maximal space of $D$ (endowed with the constructible topology) to $\insZ$, and in particular it is free (see \cite{HK-Olb-Re} and \cite{SP-scattered}).

However, not every almost Dedekind domain enjoys radical factorization; one characterization is that an almost Dedekind domain is an SP-domain if and only if it has no \emph{critical maximal ideals}, i.e., if every maximal ideal contains a finitely generated radical ideal. In order to study the case of almost Dedekind domains that are not SP-domains, in \cite{SP-scattered} the author introduced a chain $\{\inscrit_\alpha(D)\}_\alpha$ of subsets of the maximal space $\mmax$ and a chain of overrings of the base ring, constructed by applying the definition of critical ideals recursively in a manner similar to how the chain of derived sets of a topological space is defined. In this way, it is possible to recover some properties of SP-domains in a wider class of domains (the \emph{SP-scattered domains}), defined as those almost Dedekind domains such that $\inscrit_\alpha(D)$ is empty for some ordinal number $\alpha$; for example, in this case the group of invertible ideals of $D$ is free, and can be written as a direct sum of groups of continuous functions.

However, the methods used in \cite{SP-scattered} were not enough to cover all almost Dedekind domains, because they do not give any result when the set $\inscrit(D)$ of critical ideals of $D$ coincides with the maximal space of $D$. Yet, no example of this phenomenon was known; for example, an open question in \cite{hasenauer-normsets} was if there exists an almost Dedekind domain that is \emph{completely unbounded}, i.e., where the ideal function associated to every element is unbounded (see Section \ref{sect:idfunct} for the definition of the ideal function associated to an ideal).

In this paper, we show that every almost Dedekind domains has a maximal ideal that is not critical (Theorem \ref{teor:exist-noncrit}) and, as a consequence, that every almost Dedekind domain is SP-scattered (Theorem \ref{teor:SP-scat}): in particular, this shows that the results proved in \cite{SP-scattered} for SP-scattered domains actually hold in every almost Dedekind domain. 

We start in Section \ref{sect:bounded} by generalizing the notion of critical ideals to \emph{$n$-critical ideals} (for $n\inN$) and to \emph{$\omega$-bounded ideals}: the former are the maximal ideals that do not contain any finitely generated ideal whose associated ideal function is bounded by $n$, while the latter are those maximal ideal that do not contain any finitely generated ideal whose associated ideal function is bounded. Note that $1$-critical ideals are exactly critical ideals. We show that these notions can be used to define a notion analogous to the one of SP-domains, and to construct a theory that is analogous to the one developed in \cite{SP-scattered} for critical ideals.

In Section \ref{sect:anti-SP}, we show that every almost Dedekind domain has a non-critical maximal ideal (Theorem \ref{teor:exist-noncrit}) by analyzing the sets $\nu_I^{-1}((n,+\infty))$ when $I$ is a finitely generated ideal; in particular, we prove that these sets are always open (i.e., that $\nu_I$ is semicontinuous; Proposition \ref{prop:semicont}). For principal ideals, we show that if the Jacobson radical of $D$ is nonzero, then $D$ cannot be completely unbounded (Proposition \ref{prop:compl-unbound-jac}). In Section \ref{sect:SP-scattered}, we complete these results by showing that every almost Dedekind domain is SP-scattered (Theorem \ref{teor:SP-scat}) and that the set $\mmax\setminus\inscrit(D)$ is not only nonempty, but also dense in $\mmax$ (Theorem \ref{teor:dense}).

\section{Preliminaries}
\subsection{Domains}
Let $D$ be an integral domain with quotient field $K$. A \emph{fractional ideal} of $D$ is a submodule $I$ of $K$ such that $dI\subseteq D$ for some $d\in D$, $d\neq 0$; an \emph{integral ideal} is a fractional ideal contained in $D$ (i.e., an ideal of $D$), and we say that $I$ is \emph{proper} if $I\subsetneq D$.

A fractional ideal is \emph{invertible} if there is a $J$ such that $IJ=D$; in this case, $J=(D:I):=\{x\in K\mid xI\subseteq D\}$. If $I$ is invertible, we set $I^{-1}:=(D:I)$. Every invertible fractional ideal is finitely generated. The set of invertible ideals is a group under the product of ideals, denoted by $\Inv(D)$. If $T$ is a ring between $D$ and $K$ (i.e., an \emph{overring} of $D$), then the extension map $I\mapsto IT$ induces a group homomorphism $\Inv(D)\longrightarrow\Inv(T)$.

If $I$ is a fractional ideal of $D$, the \emph{$v$-closure} of $I$ is $I^v:=(D:(D:I))$; equivalently, $I^v$ is the intersection of all principal fractional ideals containing $I$. The ideal $I$ is \emph{divisorial} if $I=I^v:=(D:(D:I))$. Every invertible fractional ideal is divisorial.

An \emph{almost Dedekind domain} is a domain $D$ such that every localization $D_M$ is a discrete valuation ring; in particular, an almost Dedekind domain is one-dimensional, integrally closed, and Pr\"ufer, and thus every finitely generated ideal of an almost Dedekind domain is invertible. If $M$ is a maximal ideal of the almost Dedekind domain $D$, we say that $M$ is \emph{critical} if $M$ does not contain any nonzero radical finitely generated ideals. We denote by $\inscrit(D)$ the set of critical maximal ideals of $D$.

An \emph{SP-domain} is a domain such that every proper ideal is a product of radical ideals; an SP-domain is always an almost Dedekind domain, while an almost Dedekind domain is an SP-domain if and only if it has no critical maximal ideals. See \cite[Theorem 2.1]{olberding-factoring-SP} and \cite[Theorem 3.1.2]{fontana_factoring} for other characterizations of SP-domains.

\subsection{The inverse and the constructible topology}
Let $R$ be a unitary commutative ring. The \emph{inverse topology} on the spectrum $\Spec(R)$ of $R$ is the topology whose subbasic open sets are the sets $\V(I):=\{P\in\Spec(R)\mid I\subseteq P\}$, as $I$ ranges among the finitely generated ideals of $R$. The \emph{constructible topology} is the topology on $\Spec(R)$ generated by both the $\V(I)$ and the $\D(I):=\{P\in\Spec(D)\mid I\nsubseteq P\}$, as $I$ ranges among the finitely generated ideals of $D$. With both these topologies, $\Spec(R)$ is a compact space (it is also a \emph{spectral space}, meaning that they are homeomorphic to the prime spectrum fo a ring endowed with the Zariski topology); the constructible topology is also Hausdorff. See \cite[Chapter 1]{spectralspaces-libro} for a thorough examination of these topologies.

On the maximal spectrum $\Max(R)$ of $R$, the inverse and the constructible topology coincide \cite[Corollary 4.4.9]{spectralspaces-libro}; we denote this topological space by $\mmax$, or $\mmax_R$ if we need to underline the ring involved. Then, $\mmax$ is completely regular, totally disconnected and has a basis of clopen subsets \cite[Corollary 4.4.9]{spectralspaces-libro}. If $R$ is one-dimensional (for example, if $R$ is an almost Dedekind domain), then $\mmax$ is compact if and only if the Jacobson radical $\Jac(R)$ of $R$ is nonzero, where $\Jac(R)$ is the intersection of all the maximal ideals of $R$. In this case, the inverse and the constructible topology also agree with the Zariski topology.

\subsection{Critical ideals of higher rank}
Let $D$ be an almost Dedekind domain. The set $\inscrit(D)$ of critical maximal ideals of $D$ is a closed set of $\mmax$; by \cite[Lemma 2.3]{SP-scattered}, it follows that
\begin{equation*}
T_1:=\bigcap\{D_M\mid M\in\inscrit(D)\}
\end{equation*}
is an overring of $T_1$ whose maximal ideals are the extensions of the critical ideals of $D$. More generally, if $\alpha$ is an ordinal, we can define recursively the following sets \cite[Section 5]{SP-scattered}:
\begin{itemize}
\item $\inscrit_0(D):=\mmax$, $T_0:=D$;
\item if $\alpha=\gamma+1$ is a successor ordinal, then
\begin{equation*}
\inscrit_\alpha(D):=\{P\in\mmax\mid PT_\gamma\in\inscrit(T_\gamma)\};
\end{equation*}
\item if $\alpha$ is a limit ordinal, then
\begin{equation*}
\inscrit_\alpha(D):=\bigcap_{\gamma<\alpha}\inscrit_\gamma(D);
\end{equation*}
\item $T_\alpha:=\bigcap\{D_P\mid P\in\inscrit_\alpha(D)\}$.
\end{itemize}
Then, each $\inscrit_\alpha(D)$ is a closed set of $\mmax$, and each $T_\alpha$ is an overring of $D$ whose maximal ideals are the extensions of the ideals in $\inscrit_\alpha(D)$.

The descending chain $\{\inscrit_\alpha(D)\}_\alpha$ stabilizes; we call the minimal $\alpha$ such that $\inscrit_\alpha(D)=\inscrit_{\alpha+1}(D)$ the \emph{SP-rank} of $D$, and we set $\inscrit_\infty(D):=\inscrit_\alpha(D)$. If $\inscrit_\infty(D)=\emptyset$, we say that $D$ is \emph{SP-scattered}.

\subsection{Ideal functions}\label{sect:idfunct}
Let $D$ be an almost Dedekind domain and let $I$ be a nonzero fractional ideal of $D$. Then, $I$ induces a function
\begin{equation*}
\begin{aligned}
\nu_I\colon\mmax  & \longrightarrow \insZ,\\
M& \longmapsto v_M(I),
\end{aligned}
\end{equation*}
where $v_M$ is the valuation relative to $D_M$. This map behaves well with the ideal operations: if $I,J$ are ideals, then $\nu_{I+J}=\inf\{\nu_I,\nu_J\}$, $\nu_{IJ}=\nu_I+\nu_J$ and $\nu_{I\cap J}=\sup\{\nu_I,\nu_J\}$.

The continuity of $\nu_I$ is strongly linked with the radical factorization of $I$: indeed, if $I$ is a proper finitely generated ideal of $D$, then $\nu_I$ is continuous if and only if $I$ can be written as a finite product of radical ideals \cite[Propsition 3.10]{SP-scattered}. This fact allows to link continuous maps on $\mmax$ to the group $\Inv(D)$ of invertible ideals \cite[Sections 4 and 5]{SP-scattered}.

\section{Bounded-critical ideals}\label{sect:bounded}
If $I$ is a radical ideal of $D$, then its ideal function $\nu_I$ is just the characteristic function of the set $\V(I)$; likewise, if $\nu_I$ is the characteristic function of some subset of $\mmax$, then $I$ must be a radical ideal. Therefore, radical ideals are characterized (among proper ideals) by the fact that the supremum of their ideal functions is $1$. For these reason, in order to generalize critical ideals, we introduce the following definitions.
\begin{defin}
Let $D$ be an almost Dedekind domain and $I$ be a nonzero proper ideal. We say that $I$ is:
\begin{itemize}
\item \emph{$n$-bounded} (for some $n\inN$) if $\sup\nu_I\leq n$;
\item \emph{bounded} if $\nu_I$ is bounded.
\end{itemize}
\end{defin}

\begin{defin}
Let $D$ be an almost Dedekind domain and $M$ a maximal ideal of $D$. We say that $M$ is:
\begin{itemize}
\item \emph{$n$-critical} (for some $n\inN$) if it does not contain any $n$-bounded finitely generated ideal;
\item \emph{$\omega$-critical} (or \emph{bounded-critical}) if it does not contain any bounded finitely generated ideal.
\end{itemize}
We denote by $\critx{n}(D)$ the set of $n$-critical maximal ideals and by $\bcrit(D)$ the set of $\omega$-critical maximal ideals of $R$.
\end{defin}

The notation $\critx{n}(D)$ (instead of the perhaps more natural $\inscrit_n(D)$) is used to avoid confusion with the construction of $\inscrit_\alpha(D)$ at stage $\alpha=n$.
\begin{prop}\label{prop:critx}
Let $D$ be an almost Dedekind domain and $M$ a maximal ideal of $D$. Then, the following hold.
\begin{enumerate}[(a)]
\item\label{prop:critx:1} $\critx{1}(D)=\inscrit(D)$.
\item\label{prop:critx:n} $\{\critx{n}(D)\mid n\inN\}$ is a descending chain of closed subsets of $\mmax$.
\item\label{prop:critx:bcrit} $\bcrit(D)=\bigcap\{\critx{n}(D)\mid n\inN\}$; in particular, $\bcrit(D)$ is closed in $\mmax$.
\end{enumerate}
\end{prop}
\begin{proof}
\ref{prop:critx:1} follows from the fact that the $1$-bounded ideals are exactly the radical ideals. In \ref{prop:critx:n}, the inclusion $\critx{n+1}(D)\subseteq\critx{n}(D)$ is obvious, and thus it is enough to show that each of these sets are closed. If $M\in\mmax\setminus\critx{n}(D)$, then it contains an $n$-bounded finitely generated ideal $I$; then, $\V(I)$ is an open set of $\mmax$ containing $M$ and disjoint from $\critx{n}(D)$. Hence, $\mmax\setminus\critx{n}(D)$ is open, i.e., $\critx{n}(D)$ is closed.

\ref{prop:critx:bcrit} follows directly from the previous point.
\end{proof}

We note that, in general, the inclusion $\critx{n+1}(D)\subseteq\critx{n}(D)$ may be proper: see Example \ref{ex:ncrit} below.

Using the notion of $n$-critical ideal, it is possible to generalize part of the theory of critical ideals and SP-domains. The following two results can be seen as partial generalizations of \cite[Proposition 3.10]{SP-scattered}.
\begin{prop}\label{prop:idbounded}
Let $D$ be an almost Dedekind domain, and let $I$ be a proper ideal of $D$. Consider the following conditions:
\begin{enumerate}[(i)]
\item\label{prop:idbounded:critx} $\V(I)\cap\critx{\omega}(D)=\emptyset$;
\item\label{prop:idbounded:bounded} $I$ is bounded.
\end{enumerate}
Then, \ref{prop:idbounded:critx} $\Longrightarrow$ \ref{prop:idbounded:bounded}. If $I$ is finitely generated, then \ref{prop:idbounded:critx} $\iff$ \ref{prop:idbounded:bounded}.
\end{prop}
\begin{proof}
Suppose \ref{prop:idbounded:critx} holds. For each $M\in \V(I)$, let $J_M$ be a bounded finitely generated ideal such that $J_M\subseteq M$. (Such a $J_M$ exists since $M$ is not $\omega$-critical.) Then, $L_M:=I+J_M$ is a bounded finitely generated ideal such that $I\subseteq L_M\subseteq M$. Therefore, $\{\V(L_M)\mid M\in \V(I)\}$ is an open cover of $\V(I)$ (with respect to the constructible topology). The space $\V(I)$ is compact with respect to the constructible topology (because it is homeomorphic to the spectrum of $D/I$), and thus we can find a finite subcover $\{\V(L_{M_1}),\ldots,\V(L_{M_k})\}$. Let $L:=L_{M_1}\cdots L_{M_k}$. Then, $L$ is finitely generated and $\V(L)=\V(I)$, i.e. $\rad(L)=\rad(I)$; therefore, there is a $t$ such that $L^t\subseteq I$. However,
\begin{equation*}
\nu_{L^t}=t\nu_L=t(\nu_{L_{M_1}}+\cdots+\nu_{L_{M_k}}),
\end{equation*}
and thus $\nu_{L^t}$ is bounded; hence $\nu_I\leq\nu_{L^t}$ is bounded too. The claim is proved.

If $I$ is finitely generated, then \ref{prop:idbounded:bounded} $\Longrightarrow$ \ref{prop:idbounded:critx} follows directly from the definitions.
\end{proof}

\begin{prop}\label{prop:nfact}
Let $D$ be an almost Dedekind domain, and fix $n\inN$. Let $I$ be a proper ideal of $D$. Consider the following conditions:
\begin{enumerate}[(i)]
\item\label{prop:nfact:crit} $\V(I)\cap\critx{n}(D)=\emptyset$;
\item\label{prop:nfact:fact} there are $n$-bounded finitely generated ideals $J_1,\ldots,J_k$ such that $I=J_1\cdots J_k$.
\end{enumerate}
Then, \ref{prop:nfact:crit} $\Longrightarrow$ \ref{prop:nfact:fact}. If $I$ is finitely generated, then \ref{prop:nfact:crit} $\iff$ \ref{prop:nfact:fact}.
\end{prop}
\begin{proof}
\ref{prop:nfact:crit} $\Longrightarrow$ \ref{prop:nfact:fact} Since $\V(I)\cap\critx{n}(D)=\emptyset$, then also $\V(I)\cap\critx{\omega}(D)=\emptyset$; therefore, $\nu_I$ is bounded by Proposition \ref{prop:idbounded}. We proceed by induction on $\sup\nu_I$. If $\sup\nu_I\leq n$ the claim is trivial (just take $J_1=I$). Suppose that $\sup\nu_I>n$. Similarly to the previous proof, for each $M\in \V(I)$ let $J_M$ be an $n$-bounded finitely generated ideal such that $J_M\subseteq M$, and define $L_M:=I+J_M$; then, $L_M$ is a $n$-bounded finitely generated ideal such that $I\subseteq L_M\subseteq M$. The family $\{\V(L_M)\mid M\in \V(I)\}$ is an open cover of $\V(I)$ (with respect to the constructible topology), and since $\V(I)$ is compact we can find a finite subcover $\{\V(L_{M_1}),\ldots,\V(L_{M_k})\}$. Let $L:=L_{M_1}\cap\cdots\cap L_{M_k}$. Then, $L$ is a finitely generated ideal such that $\sup\nu_L\leq n$, $I\subseteq L$ and $\V(I)=\V(L)$. Set $I_1:=IL^{-1}$: then, $I_1$ is again a proper ideal of $D$, and
\begin{equation*}
\nu_{I_1}(M)=\nu_I(M)-\nu_L(M)\leq\nu_I(M)-1
\end{equation*}
for every $M\in \V(I)$. Therefore, $\sup\nu_{I_1}<\sup\nu_I$, and by induction $I_1$ has a factorization $I_1=J_1\cdots J_t$ as product of $n$-bounded ideals. Hence, $I=I_1L=J_1\cdots J_tL$ has this kind of factorization too. The claim is proved.

Suppose that $I$ is finitely generated and that \ref{prop:nfact:fact} holds. Let $M\in \V(I)$. Since $J_1\cdots J_k\subseteq M$, there is an $i$ such that $J_i\subseteq M$; since $I$ is invertible, also $J_i$ is invertible and thus finitely generated. By definition, $M$ cannot be $n$-critical. Hence $\V(I)\cap\critx{n}(D)=\emptyset$.
\end{proof}

\begin{oss}
The equivalence of the two conditions does not hold if $I$ is not finitely generated. Indeed, if $M$ is $\omega$-critical, then $M$ is bounded, but $\V(M)\cap\critx{\omega}(D)=\{M\}\neq\emptyset$. For the same reason, if $n\inN$, then $M$ has a factorization as product of $n$-critical ideals (namely, $M$ itself) but $\V(M)\cap\critx{n}(D)=\{M\}\neq\emptyset$.
\end{oss}

If we use these results on all ideals, we obtain equivalences similar to the ones characterizing SP-domains.
\begin{cor}\label{cor:SP-bounded}
Let $D$ be an almost Dedekind domain. Then, the following are equivalent:
\begin{enumerate}[(i)]
\item\label{cor:SP-bounded:critx} $\critx{\omega}(D)=\emptyset$;
\item\label{cor:SP-bounded:princ} every nonzero principal proper ideal is bounded;
\item\label{cor:SP-bounded:fg} every nonzero finitely generated proper ideal is bounded;
\item\label{cor:SP-bounded:id} every nonzero proper ideal is bounded;
\item\label{cor:SP-bounded:fract} $\nu_I$ is bounded for every fractional ideal $I$.
\end{enumerate}
\end{cor}
\begin{proof}
\ref{cor:SP-bounded:critx} $\iff$ \ref{cor:SP-bounded:fg} follows from Proposition \ref{prop:idbounded}. \ref{cor:SP-bounded:fract} $\Longrightarrow$ \ref{cor:SP-bounded:id} $\Longrightarrow$ \ref{cor:SP-bounded:fg} $\Longrightarrow$ \ref{cor:SP-bounded:princ} is obvious. To show \ref{cor:SP-bounded:princ} $\Longrightarrow$ \ref{cor:SP-bounded:id}, let $I\neq(0)$ be an ideal. For every $i\in I\setminus\{0\}$, $\nu_{iD}$ is bounded; since $\nu_I\leq\nu_{iD}$, also $\nu_I$ is bounded. Finally, if \ref{cor:SP-bounded:id} holds, and $I\in\insfracid(D)$, then $I=JL^{-1}$ for some proper ideals $J,L$ of $D$, with $L$ finitely generated \cite[Proposition 3.3]{SP-scattered}, and thus $\nu_I=\nu_J-\nu_L$ is bounded.
\end{proof}

\begin{cor}\label{cor:nSP-bounded}
Let $D$ be an almost Dedekind domain and fix an $n\inN$. Then, the following are equivalent:
\begin{enumerate}[(i)]
\item\label{cor:nSP-bounded:critx} $\critx{n}(D)=\emptyset$;
\item\label{cor:nSP-bounded:fg} every nonzero finitely generated proper ideal is the product of $n$-bounded ideals;
\end{enumerate}
\end{cor}
\begin{proof}
Follows from Proposition \ref{prop:nfact}.
\end{proof}

A group strongly corrected with $\Inv(D)$ is the group of divisorial ideals of $D$. When $D$ is an almost Dedekind domain (or, more generally, a one-dimensional Pr\"ufer domain), the set $\Div(D)$ of divisorial ideals is a group under the $v$-product $I\times_vJ:=(IJ)^v=(D:(D:IJ))$. Both $\Inv(D)$ and $\Div(D)$ are $\ell$-groups under the opposite of the containment order (i.e., $I\leq J$ if and only if $I\supseteq J$), and from this point of view $\Div(D)$ is the completion of $\Inv(D)$ (see \cite{darnel-lgroups} for background on $\ell$-groups and \cite[Section 3]{HK-Olb-Re} and \cite[Section 6]{SP-scattered} for discussion of $\Inv(D)$ and $\Div(D)$ as $\ell$-groups). Using $\omega$-critical ideals we can prove the following partial generalization of \cite[Theorem 6.5]{SP-scattered}.
\begin{prop}\label{prop:DivD}
Let $D$ be an almost Dedekind domain such that $\critx{\omega}(D)=\emptyset$. Then, $\Div(D)$ is a free group.
\end{prop}
\begin{proof}
Let $\mathcal{F}_b(\mmax,\insZ)$ be the group of all bounded functions $\mmax\longrightarrow\insZ$. Since $\critx{\omega}(D)$ is empty, by Corollary \ref{cor:SP-bounded} every finitely generated fractional ideal is invertible; therefore, the assignment $I\mapsto\nu_I$ defines a map $\Inv(D)\longrightarrow\mathcal{F}_b(\mmax,\insZ)$ that is also a $\ell$-group homomorphism. Since $\mathcal{F}_b(\mmax,\insZ)$ is a complete $\ell$-group, it contains the completion of $\Inv(D)$, i.e., it contains a subgroup isomorphic to $\Div(D)$. However, since $\mathcal{F}_b(\mmax,\insZ)$ is a free group \cite[Satz 1]{nobeling}, this implies that $\Div(D)$ is free too. The claim is proved.
\end{proof}

\begin{oss}
Proposition \ref{prop:DivD} shows the existence of a map $\phi:\Div(D)\longrightarrow\mathcal{F}_b(\mmax,\insZ)$ such that $\Phi(\Div(D))$ is isomorphic to $\Div(D)$. This map does \emph{not} coincide with the map associating to each $I\in\Div(D)$ its ideal function $\nu_I$ (which belongs to $\mathcal{F}_b(\mmax,\insZ)$ by Corollary \ref{cor:SP-bounded}): indeed, if $I$ is divisorial ideal but not finitely generated, then $-\nu_I$ is not in the form $\nu_J$ for any ideal $J$ (otherwise $IJ=D$ and $I$ would be invertible). In particular, if $L=(D:I)$ is the inverse of $I$ in $\Div(D)$, then $\nu_L\neq-\nu_I$. Therefore, the assignment $I\mapsto\nu_I$ is not a group homomorphism when seen as a map $\Div(D)\longrightarrow\mathcal{F}_b(\mmax,\insZ)$.
\end{oss}

Fix now an $n\inN\cup\{\omega\}$. As with critical ideals, it is possible to recursively define a sequence $\{\critx{n}_\alpha(D)\}_\alpha$ of subsets of $\mmax$ and a sequence $\{\phantom{}^nT_\alpha\}_\alpha$ of overrings in the following way:
\begin{itemize}
\item $\critx{n}_0(D):=\mmax$, $\phantom{}^nT_0:=D$;
\item if $\alpha=\gamma+1$ is a successor ordinal, then
\begin{equation*}
\critx{n}_\alpha(D):=\{P\in\mmax\mid PT_\gamma\in\critx{n}(\phantom{}^nT_\gamma)\};
\end{equation*}
\item if $\alpha$ is a limit ordinal, then
\begin{equation*}
\critx{n}_\alpha(D):=\bigcap_{\gamma<\alpha}\critx{n}_\gamma(D);
\end{equation*}
\item $\phantom{}^nT_\alpha:=\bigcap\{D_P\mid P\in\critx{n}_\alpha(D)\}$.
\end{itemize} 
It is easy to see that $\{\critx{n}_\alpha(D)\}_\alpha$ is a descending sequence of closed subsets of $\mmax$, while $\{\phantom{}^nT_\alpha\}_\alpha$ is an ascending sequence of overrings such that $\mmax_{\phantom{}^nT_\alpha}$ corresponds to $\critx{n}_\alpha(R)$, in the sense that the maximal ideals of $\phantom{}^nT_\alpha$ are the extensions of the elements of $\critx{n}_\alpha(R)$.

The sequence $\{\critx{n}_\alpha(D)\}_\alpha$ must stabilize; we call the minimal ordinal $\alpha$ such that $\critx{n}_\alpha(D)=\critx{n}_{\alpha+1}(D)$ the \emph{$n$-SP-rank} of $D$. In this case, we denote by $\critx{n}_\infty(D)$ the set $\critx{n}_\alpha(D)$; equivalently, $\critx{n}_\infty(D)$ is just the intersection of all $\critx{n}_\alpha(D)$.

Using this sequence, it is possible to mimic the reasoning and the results in \cite[Section 5]{SP-scattered} using $n$-critical or bounded-critical ideals instead of simply critical ideals; indeed, the most delicate point in the proof of \cite[Proposition 5.8]{SP-scattered} is the fact that the kernel of the natural extension map $\Psi_\gamma:\Inv(T_\gamma)\longrightarrow\Inv(T_{\gamma+1})$ is isomorphic to a set of continuous maps of compact support, namely to $\mathcal{C}_c(\inscrit_\gamma(R)\setminus\inscrit_{\gamma+1}(R),\insZ)$, and, since all these maps are bounded, the kernel is free. The same property holds also for $n$-critical ideals.
\begin{prop}\label{prop:freekernel}
Let $D$ be an almost Dedekind domain, $n\inN\cup\{\omega\}$. For every ordinal $\gamma$, the kernel of the map $\Psi_\gamma:\Inv(\phantom{}^nT_\gamma)\longrightarrow\Inv(\phantom{}^nT_{\gamma+1})$ is free.
\end{prop}
\begin{proof}
By definition, $\phantom{}^nT_{\gamma+1}$ is just the first element of the sequence of overrings if we use $\phantom{}^nT_\gamma$ instead of $D$ as a starting point. Hence it is enough to prove the claim for $\gamma=0$, i.e., for the extension $\Psi^{(n)}:\Inv(D)\longrightarrow\Inv(\phantom{}^nT_1)$. Moreover, $\phantom{}^nT_1$ is contained in $\phantom{}^\omega T_1$ (since $\critx{\omega}(D)\subseteq\critx{n}(D)$ for every $n$) and thus the kernel of $\Psi^{(n)}$ is contained in the kernel of $\Psi^{(\omega)}$. Since a subgroup of a free abelian group is free, it is enough to show the statement for $n=\omega$.

The kernel of $\Psi^{(\omega)}$ is the set of invertible ideals $I$ such that $ID_M=D_M$ for every $M\in\critx{\omega}(D)$. If $I\subseteq D$ then by Proposition \ref{prop:idbounded}, this condition is equivalent to $I$ being bounded. If $I\nsubseteq D$, then by \cite[Proposition 3.3]{SP-scattered} we have $I=JL^{-1}$ where $J,L$ are proper ideals such that $\V(J)$ and $\V(L)$ are disjoint from $\critx{\omega}(D)$. Hence $\nu_I$ is bounded, since it is the difference of two bounded function. It follows that $\ker\Psi^{(\omega)}$ is isomorphic to a subgroup of the set of all bounded functions $\mmax\setminus\critx{\omega}(D)\longrightarrow \insZ$, which is free by \cite[Satz 1]{nobeling}. The claim is proved.
\end{proof}

The rest of the reasoning in \cite[Section 5]{SP-scattered} can be transposed to the case of $n$-critical ideals; hence, for every ordinal $\alpha$ we obtain an exact sequence
\begin{equation*}
0\longrightarrow\bigoplus_{i<\alpha}\phantom{}^n\Delta_i\longrightarrow\Inv(D)\longrightarrow\Inv(\phantom{}^nT_\alpha)\longrightarrow 0,
\end{equation*}
where $\phantom{}^n\Delta_i$ is the kernel of $\Inv(\phantom{}^nT_i)\longrightarrow\Inv(\phantom{}^nT_{i+1})$ and thus free by Proposition \ref{prop:freekernel}. In particular, if $\phantom{}^nT_\alpha=K$ (equivalently, if $\critx{n}_\alpha(D)=\emptyset$), then we have
\begin{equation*}
\Inv(D)\simeq\bigoplus_{i<\alpha}\phantom{}^n\Delta_i,
\end{equation*}
and thus $\Inv(D)$ is free. However, we shall see in the subsequent part of the paper that this fact does not improve the results in \cite{SP-scattered}, in the sense that we do not obtain the freeness of $\Inv(D)$ for any ``new'' class of rings.

\begin{ex}\label{ex:ncrit}
In general, it may be that $\critx{n}(D)\subsetneq\inscrit(D)$. We use the same example of \cite[Example 3.4.1]{fontana_factoring}: let $X_1,\ldots,X_n,\ldots$ be a countable set of indeterminates over a field $F$, let $Y_n:=\prod_{i\geq n+1}X_i^2$, and let
\begin{equation*}
D_n:=F[X_1,\ldots,X_n,Y_n]_{(Y_n)}\cap\bigcap_{i=1}^nF[X_1,\ldots,X_n,Y_n]_{(X_i)}
\end{equation*}
and $D:=\bigcup_nD_n$. Then, $D$ is an almost Dedekind domain with quotient field $F(Y_0,X_1,\ldots,X_n,\ldots)$ and whose maximal ideals are the principal ideals $M_k:=X_kD$ and the (non-finitely generated) ideal $M$ generated by all the $Y_k$.

Then, $\inscrit(D)=\{M\}$; however, $\nu_{Y_t}(M)=1$ and $\nu_{Y_t}(M_k)=2$ if $t\geq k$. In particular, $M$ is critical but not $2$-critical, so that $\critx{2}(D)=\emptyset$.

The same example can be slightly modified to obtain different sets of critical ideals.

If we define $Y_n:=\prod_{i\geq n+1}X_i^r$ (that is, using the constant exponent $r$ instead of $2$, but leaving the rest of the construction unchanged) one similarly has that $M$ is critical, $2$-critical, \ldots, $(r-1)$-critical, but not $r$-critical. Thus $\critx{n}(D)=\{M\}$ if $n<r$ and $\critx{n}(D)=\emptyset$ if $n\geq r$ or $n=\omega$.

If we define $Y_n:=\prod_{i\geq n+1}X_i^i$ (as in \cite[Example 3.4.2]{fontana_factoring}) then the principal ideal $Y_0D$ is unbounded. Hence $\inscrit(D)=\critx{n}(D)=\critx{\omega}(D)=\{M\}$ for every $M$.
\end{ex}

\begin{quest}
Is there an almost Dedekind domain $D$ such that $\critx{n+1}(D)\subsetneq\critx{n}(D)$ for every $n\inN$?
\end{quest}

\section{Anti-SP domains}\label{sect:anti-SP}
\begin{defin}
Let $D$ be an almost Dedekind domain. We say that $D$ is an \emph{anti-SP domain} if $\Max(D)=\inscrit(D)$.
\end{defin}

By \cite[Proposition 3.10]{SP-scattered}, this condition is equivalent to the fact that no finitely generated ideal has radical factorization; thus, anti-SP domains are almost Dedekind domains that are ``as far as possible'' from being SP-domains. Under this condition, the procedure used in \cite[Section 5]{SP-scattered} to study $\Inv(D)$ fails at the first stage; one could hope that using $n$-bounded or $\omega$-bounded ideals one can improve the results in \cite{SP-scattered}. This is not true, as the next proposition shows.
\begin{prop}\label{prop:inscrit-full}
Let $D$ be an almost Dedekind domain. Then, the following are equivalent:
\begin{enumerate}[(i)]
\item\label{prop:inscrit-full:full} $\inscrit(D)=\Max(D)$;
\item\label{prop:inscrit-full:n} $\critx{n}(D)=\Max(D)$ for some $n\inN$;
\item\label{prop:inscrit-full:nall} $\critx{n}(D)=\Max(D)$ for all $n\inN$;
\item\label{prop:inscrit-full:critx} $\critx{\omega}(D)=\Max(D)$;
\item\label{prop:inscrit-full:unbounded} no finitely generated ideal is bounded.
\end{enumerate}
\end{prop}
\begin{proof}
\ref{prop:inscrit-full:critx} and \ref{prop:inscrit-full:unbounded} are equivalent by definition.

\ref{prop:inscrit-full:critx} $\Longrightarrow$ \ref{prop:inscrit-full:nall} $\Longrightarrow$ \ref{prop:inscrit-full:n} $\Longrightarrow$\ref{prop:inscrit-full:full} follows from the containments $\critx{\omega}(D)\subseteq\critx{n}(D)\subseteq\inscrit(D)$.

\ref{prop:inscrit-full:full} $\Longrightarrow$ \ref{prop:inscrit-full:unbounded} Suppose $I$ is bounded, and let $J:=\rad(I)$; then, $J^n\subseteq I$ for some $n\inN$.  Hence, $(J^n)^v\subseteq I^v=I$; however,
\begin{equation*}
(J^n)^v=(J\cdots J)^v=(J^v\cdots J^v)^v=((J^v)^n)^v
\end{equation*}
and thus $J^v\neq D$. Therefore, there is an $\alpha\in K$ such that $J\subseteq D\cap\alpha D\subsetneq D$. This implies that $L:=D\cap\alpha D$ is a finitely generated ideal that is also radical (since it contains the radical ideal $J$); in particular, every maximal ideal containing $L$ is not critical, and $\inscrit(D)\neq\Max(D)$.
\end{proof}

We now want to show that there are actually no anti-SP domains.
\begin{prop}\label{prop:semicont}
Let $D$ be an almost Dedekind domain and let $I$ be a proper finitely generated ideal of $D$. Then, $\nu_I$ is lower semicontinuous, with respect to the constructible topology.
\end{prop}
\begin{proof}
We need to show that $\nu_I^{-1}([0,n])=\{M\in\Max(D)\mid v_M(I)\leq n\}$ is closed for every $n\inN$. We first note that $\nu_I^{-1}(0)=\D(I)$ is closed in the constructible topology since $I$ is finitely generated; hence, we only need to consider $X_n:=\nu_I^{-1}([1,n])$.

Let $J:=\rad(I)$ and let $L:=I^{-1}J^{n+1}$. Take a maximal ideal $M$: if $M\notin \V(J)$, then $v_M(L)=0$ since $v_M(I)=0=v_M(J)$. On the other hand, if $M\in \V(J)$, then $v_M(L)=(n+1)v_M(J)-v_M(I)=(n+1)-v_M(I)$. Hence, $v_M(L)>0$ if and only if $v_M(I)<n+1$, i.e., if and only if $M\in X_n$. It follows that $\V(L\cap D)=X_n$ is closed in the Zariski topology. However, the Zariski, inverse and constructible topology agree on $\V(I)$, and thus (since $X_n\subseteq\V(I)$ and $\V(I)$ is closed in $\mmax$) it follows that $X_n$ is also closed in the constructible topology. Thus, $\nu_I^{-1}([0,n])=\D(I)\cup X_n$ is closed in the constructible topology.
\end{proof}

\begin{lemma}\label{lemma:dense}
Let $D$ be an anti-SP domain with nonzero Jacobson radical. Let $I\subseteq\Jac(D)$ be a nonzero finitely generated ideal. For every $n\inN$, the set 
\begin{equation*}
Y_n:=\nu_I^{-1}((n,+\infty))=\{M\in\Max(D)\mid v_M(I)>n\}
\end{equation*}
is dense in $\mmax$.
\end{lemma}
\begin{proof}
Suppose that $Y_n$ is not dense: then, there is a finitely generated ideal $J$ such that the clopen set $\V(J)$ does not meet the closure of $Y_n$. Let $J':=J+I$: then, $J'\neq D$ since every maximal ideal containing $J$ contains also $J'$ (since $I\subseteq\Jac(D)$). Let $M\in\mmax$: then, $v_M(J')=0$ if $M\in Y_n$ (since $M\notin\V(J)$), while $v_M(J')\leq v_M(I)\leq n$ if $M\notin Y_n$. Hence, $\nu_{J'}$ is bounded by $n$, against the fact that $D$ is anti-SP. Thus $Y_n$ is dense in $\mmax$.
\end{proof}

\begin{teor}\label{teor:exist-noncrit}
Let $D$ be an almost Dedekind domain that is not a field. Then, $\inscrit(D)\neq\Max(D)$. In particular, no almost Dedekind domain is an anti-SP domain.
\end{teor}
\begin{proof}
Suppose first that $\Jac(D)\neq(0)$, and suppose that $\inscrit(D)=\Max(D)$: then, by Proposition \ref{prop:inscrit-full}, $\nu_I$ is not bounded for every proper finitely generated ideal of $D$.

Let $I\subseteq\Jac(D)$ be a finitely generated ideal of $D$, and let $Y_n:=\nu_I^{-1}((n,+\infty))=\{M\in\Max(D)\mid v_M(I)>n\}$. By Proposition \ref{prop:semicont}, each $Y_n$ is open in $\mmax$, and by Lemma \ref{lemma:dense} they are dense in $\mmax$. Therefore, $\{Y_n\}_{n\inN}$ is a family of open and dense subsets of the compact Hausdorff space $\mmax$; by the Baire category theorem, its intersection must be dense, and in particular nonempty. However, if $M\in\bigcap_nY_n$, then $v_M(I)>n$ for every $n$; since this cannot happen, it follows that we cannot have $\inscrit(D)=\Max(D)$.

Suppose now that $\Jac(D)=(0)$, and let $I$ be a finitely generated ideal. Then, $T:=(1+I)^{-1}D=\bigcap\{D_M\mid M\in\V(I)\}$ is an almost Dedekind domain with nonzero Jacobson radical (indeed, $I\subseteq\Jac(T)$) and thus there is a proper finitely generated ideal $J$ of $T$ such that $\nu_J:\mmax_T\longrightarrow\insN$ is bounded. Let $J_0$ be a finitely generated ideal of $D$ such that $J=J_0T$, and let $L:=J_0+I$. Then, $\nu_L(M)=v_M(L)=0$ if $M\notin\V(I)$, while $\nu_L(M)=v_M(L)\leq v_{MT}(J)=\nu_J(MT)$ if $M\in\V(I)$. Furthermore, $L\neq D$ since $LT$ is contained in every maximal ideal of $T$ containing $J$. Therefore, $\nu_L$ is a nonzero bounded function. By Proposition \ref{prop:inscrit-full}, $\inscrit(D)\neq\Max(D)$.
\end{proof}

The previous results only deal with finitely generated ideals; however, it would be interesting to know if we can also extend those results to \emph{principal} ideals. An almost Dedekind domain is said to be \cite{hasenauer-normsets}:
\begin{itemize}
\item \emph{bounded} if every principal ideal is bounded;
\item \emph{completely unbounded} if no proper principal ideal is bounded.
\end{itemize}

Bounded almost Dedekind domains are characterized by Corollary \ref{cor:SP-bounded}: $D$ is bounded if and only if $\critx{\omega}(D)=\emptyset$. The case of completely unbounded domains is more delicate; we cannot answer in general the question posed in \cite{hasenauer-normsets} on the existence of such domains, but we can answer negatively in two cases, the first of which is obvious in light of the previous results.
\begin{prop}
Let $D$ be a B\'ezout almost Dedekind domain. Then, $D$ is not completely unbounded.
\end{prop}
\begin{proof}
If $D$ is B\'ezout, every finitely generated ideal is principal. The claim follows from Theorem \ref{teor:exist-noncrit}.
\end{proof}

\begin{prop}\label{prop:compl-unbound-jac}
Let $D$ be an almost Dedekind domain with $\Jac(\phantom{}^{\omega}T_1)\neq(0)$. Then, $D$ is not completely unbounded.
\end{prop}
\begin{proof}
Let $I:=\bigcap\{P\mid P\in\critx{\omega}(D)\}$. Then, $I=\Jac(\phantom{}^{\omega}T_1)\cap D$, and thus $I\neq(0)$; moreover, $\V(I)=\critx{\omega}(D)$. By Theorem \ref{teor:exist-noncrit}, $\inscrit(D)\neq\mmax$, and thus $I\neq\Jac(D)$; let $b\in I\setminus\Jac(D)$. Then, there is an $r\in D$ such that $a:=1-rb$ is a nonunit of $D$; moreover, by construction, $a$ is not contained in any $\omega$-critical ideal. By Proposition \ref{prop:idbounded}, $\nu_{aD}$ is bounded. The claim is proved.
\end{proof}

\begin{cor}
Let $D$ be an almost Dedekind domain such that $|\inscrit(D)|<\infty$. Then, $D$ is not completely unbounded.
\end{cor}
\begin{proof}
If $\inscrit(D)$ is finite, then so is $\critx{\omega}(D)$, and thus $\phantom{}^\omega  T_1$ is semilocal, and in particular its Jacobson radical is nonzero. The claim follows from Proposition \ref{prop:compl-unbound-jac}.
\end{proof}

\begin{cor}\label{cor:Jac-complunbounded}
Let $D$ be an almost Dedekind domain with $\Jac(D)\neq(0)$. Then, $D$ is not completely unbounded.
\end{cor}
\begin{proof}
If $\Jac(D)\neq(0)$, then also $\Jac(\phantom{}^\omega T_1)\neq(0)$. The claim follows from Proposition \ref{prop:compl-unbound-jac}.
\end{proof}

Recall that a domain is \emph{atomic} if every nonzero nonunit can be written as a finite product of irreducible elements. 
\begin{cor}
Let $D$ be an almost Dedekind domain that is not Dedekind. If $D$ is atomic, then $\Jac(D)=(0)$.
\end{cor}
\begin{proof}
By \cite[Theorem 3.15]{hasenauer-normsets}, there is an overring $T$ of $D$ that is atomic and completely unbounded. If $\Jac(D)\neq(0)$, then also $\Jac(T)\neq(0)$, contradicting Corollary \ref{cor:Jac-complunbounded}.
\end{proof}

\section{SP-scatteredness}\label{sect:SP-scattered}
The fact that every almost Dedekind domain has a noncritical maximal ideal has a very powerful consequence for SP-scattered domains.
\begin{teor}\label{teor:SP-scat}
Let $D$ be an almost Dedekind domain. Then, $D$ is SP-scattered.
\end{teor}
\begin{proof}
Let $\alpha$ be an the SP-rank of $D$, and let $T_\alpha:=\bigcap\{D_M\mid M\in\inscrit_\alpha(D)\}$: then, the maximal ideals of $T_\alpha$ are extensions of the members of $\inscrit_\alpha(D)$. Thus, $\inscrit(T_\alpha)=\Max(T_\alpha)$, since otherwise the set $\inscrit_{\alpha+1}(D)$ (which are the restrictions of the critical ideals of $T_\alpha$) would be strictly contained in $\inscrit_\alpha(D)$. By Theorem \ref{teor:exist-noncrit}, it follows that $T_\alpha$ must be a field, i.e., that $D$ is SP-scattered.
\end{proof}

\begin{cor}
Let $D$ be an almost Dedekind domain. Then, $\inscrit_\infty(D)=\critx{n}_\infty(D)=\critx{\omega}_\infty(D)=\emptyset$ for every $n\inN$.
\end{cor}
\begin{proof}
We have $\critx{\omega}_\infty(D)\subseteq\inscrit_\infty(D)$, and by the previous proposition $\inscrit_\infty(D)=\emptyset$. The claim is proved.
\end{proof}

In particular, all results about SP-scattered domains do actually hold for every almost Dedekind domain; we denote by $\mathcal{C}_c(X,\insZ)$ the group of continuous functions of bounded support on a topological space $X$.
\begin{prop}\label{prop:Invfree}
\cite[Theorem 5.9]{SP-scattered} Let $D$ be an almost Dedekind domain with SP-rank $\alpha$. Then,
\begin{equation*}
\Inv(D)\simeq\bigoplus_{i<\alpha}\mathcal{C}_c(\inscrit_i(D)\setminus\inscrit_{i+1}(D),\insZ);
\end{equation*}
in particular, $\Inv(D)$ is free.
\end{prop}

\begin{prop}
{\rm(\cite[Corollary 7.5]{SP-scattered}, \cite[Proposition 7.5]{radical-semistar})}
Let $D$ be an almost Dedekind domain, let $\ell$ be a singular length function on $D$ and let $\tau$ be the corresponding ideal colength. Then, $\ell$ is radical, i.e.,
\begin{equation*}
\tau(I)=\tau(\rad(I))
\end{equation*}
for every ideal $I$.
\end{prop}

To conclude, we improve Theorem \ref{teor:exist-noncrit} by showing that the set $\mmax\setminus\inscrit(D)$ is not only nonempty, but also dense in $\mmax$; we need the following definition.

\begin{defin}
Let $D$ be an almost Dedekind domain and $P$ be a maximal ideal. The \emph{SP-height} of $P$ is the ordinal $\alpha$ such that $P\in X_\alpha:=\inscrit_\alpha(D)\setminus\inscrit_{\alpha+1}(D)$; equivalently, the SP-height of $P$ is the maximal ordinal $\alpha$ such that $P\in\inscrit_\alpha(D)$.
\end{defin}
Note that this definition is well-posed, since $\bigcup_\alpha X_\alpha=\mmax\setminus\inscrit_\infty(D)=\emptyset$, by Theorem \ref{teor:SP-scat}.

\begin{teor}\label{teor:dense}
Let $D$ be an almost Dedekind domain. Then, the set $\mmax\setminus\inscrit(D)$ is dense in $\mmax$.
\end{teor}
\begin{proof}
Set $Z:=\mmax\setminus\inscrit(D)$. Let $P$ be a prime ideal of $D$; we proceed by induction on the SP-height of $P$.

If the height is $0$, then $P\in Z\subseteq\overline{Z}$. Suppose that $\alpha>0$ (so $P$ is not critical) and that, for every prime ideal $Q$ of SP-height $\beta<\alpha$, we have $Q\in\overline{Z}$; then, $\overline{Z}$ contains $Z_\alpha:=\bigcup_{\beta<\alpha}X_\beta$. Suppose $P$ has SP-height $\alpha$, i.e., $P\in X_\alpha$. By definition, $PT_\alpha$ is not critical in $T_\alpha$, and thus there is a finitely generated ideal $I$ of $D$ such that $IT_\alpha$ is radical in $T_\alpha$ and $IT_\alpha\subseteq PT_\alpha$ (thus, $I\subseteq P$). If $P\notin\overline{Z}$, there is an open subset $\Omega$ of $\mmax$ containing $P$ but disjoint from $\overline{Z}$ and thus from $Z_\alpha$; since the $\V(J)$ (as $J$ ranges among the finitely generated ideals of $D$) are a subbasis of $\mmax$, we can suppose that $\Omega=\V(J)$. Let $I':=I+J$: then, $I'$ is finitely generated, $I'\subseteq P$, and $\V(I')\cap Z_\alpha\subseteq\V(J)\cap Z_\alpha=\emptyset$. Moreover, if $M\notin Z_\alpha$, then $\nu_{I'}(M)\leq\nu_I(M)=\nu_{IT_\alpha}(M)\leq 1$; hence, $I'$ is a radical ideal. This would imply that $P$ is a non-critical ideal of $D$, a contradiction; hence $P$ must belong to $\overline{Z}$. By induction, $\overline{Z}=\mmax$, i.e., $Z$ is dense in $\mmax$.
\end{proof}

\bibliographystyle{plain}
\bibliography{/bib/articoli,/bib/libri,/bib/miei}
\end{document}